\documentclass[12pt,reqno,twoside]{amsart}
\usepackage{amsmath, amssymb,latexsym}
\usepackage[all,cmtip]{xy}

\newtheorem{thm}{Theorem}[section]
\newtheorem{lem}[thm]{Lemma}
\newtheorem{prop}[thm]{Proposition}

\newtheorem{defn}[thm]{Definition}

\newtheorem{obs}[thm]{Observation}

\numberwithin{equation}{section}

\newcommand{\N}{{\mathbb{N}}}

\begin{document}
\title[Minimal Pairs]{Nonexistence of Minimal Pairs for Generic Computability}
\author{Gregory Igusa}

\maketitle

\begin{abstract}

A generic computation of a subset $A$ of $\mathbb{N}$ consists of a a computation that correctly computes \it most \rm of the bits of $A$, and which never incorrectly computes any bits of $A$, but which does not necessarily give an answer for every input. The motivation for this concept comes from group theory and complexity theory, but the purely recursion theoretic analysis proves to be interesting, and often counterintuitive. The primary result of this paper is that there are no minimal pairs for generic computability, answering a question of Jockusch and Schupp.

\end{abstract}

\section{Introduction}

In a recent paper, Jockusch and Schupp \cite{stuff} introduce and analyze a number of notions of almost computability. Following their notation, we make the following definitions:

\begin{defn}
\rm Let $A$ be a subset of the natural numbers. Then $A$ has \it{density 1} \rm if the limit of the densities of its initial segments is 1, or in other words, if $\lim_{n\rightarrow\infty}\frac{|{A\upharpoonright n}|}{n}=1$. In this case we will frequently say that $A$ is \it{density-1.} \rm
\end{defn}

In this paper, a subset of the natural numbers is often referred to as a real. Note that the intersection of two reals is density-1 if and only if each of the reals is density-1. Binary branching trees and other countable sets will sometimes be referred to as reals, although the density of a real will never be mentioned unless it is explicitly described as a subset of $\mathbb{N}$.

\begin{defn}
\rm A real $A$ is \it{generically computable} \rm if there exists a partial recursive function $\varphi$ whose domain has density 1 such that if $\varphi(n)=1$ then $n\in A$, and if $\varphi(n)=0$ then $n\notin A$.
\end{defn}

\begin{defn}
\rm For reals $A$ and $B$, $A$ is \it{generically B-computable} \rm if $A$ is generically computable using $B$ as an oracle. In this case, we frequently say B \it{generically computes} \rm A.
\end{defn}

Note that this notion of computation yields a binary relation on reals which is highly nontransitive, and which should not be confused by the transitive relation, \it generic reducibility\rm, in which only an enumerably presented density-1 subset of the information in the oracle can be used. (Generic reducibility will be defined and discussed in more detail in Section 3.)

In fact, generic computation is so far from being transitive that any countable reflexive binary relation embeds in the reals under this relation. To prove this, we first make three observations:



\begin{obs}\label{obs1}
There exists a recursive reflexive binary relation $R$ on $\mathbb{N}$ such that for any reflexive binary relation $R^\prime$ with countable domain, $R^\prime$ is isomorphic to $R$ restricted to some subset of $\mathbb{N}$.
\end{obs}

The relation $R$ is simply the Fraisse limit of the finite reflexive binary relations. We also present a direct construction of $R$:

\begin{proof}
We build our relation $R$ in stages. At the end of any stage, $R$ will be defined on a finite initial segment of $\mathbb{N}$.

At stage 0, the domain of $R$ is $\lbrace0\rbrace$, and $\langle0,0\rangle\in R$

At stage $s+1$, we extend $R$ to be defined on the next $4^n$ many elements of $\mathbb{N}$, where $n$ is the domain of $R$ at the end of stage $s$. For every possible way to extend $R\upharpoonright n$ by one element, we take one of the new elements and define $R$ on on that element via the chosen extension. Every new element is related to itself (to maintain reflexivitiy), and no new element is related to any other new element. (Note that since the relationship is not symmetric, we need $4^n$ many new elements. For every old element $a$, we must choose whether a new element $b$ satisfies $aRb$ and also whether it satisfies $bRa$.)

This construction can clearly be carried out recursively. Any countable reflexive binary relation $R^\prime$ embeds in it by picking a counting $\langle a_i\rangle$ of the domain of $R^\prime$, and mapping inductively each $a_i$ to an element that was added to the domain of $R$ at stage $i$. (By the construction of $R$, when the embedding has been defined on $a_0,...,a_n$, there is exists a way to extend the embedding to map $a_{n+1}$ to an element added at stage $n+1$.)
\end{proof}

\begin{obs}\label{R}
For any real $A$, there is another real $\mathcal{R}(A)$ such that $A$ generically computes $\mathcal{R}(A)$ and $A$ can be computed from any generic description of $\mathcal{R}(A)$.
\end{obs}

There are many different constructions of $\mathcal{R}(A)$, any of which will suffice for the purposes of this proposition. An example would be coding the entries of $A$ into the columns of $\mathcal{R}(A)$, ie $n\in \mathcal{R}(A)\leftrightarrow m\in A$, where $2^m$ is the largest power of $2$ dividing $n$. A more thorough proof will be presented in Section 3, where the specifics of the definition of $\mathcal{R}(A)$ will become more important.

\begin{obs}\label{obs3}
There exists a countable sequence of reals $\langle X_i\rangle$ such that for each $i$, $X_i$ cannot be computed by the recursive join of the rest of the $X_j$.
\end{obs}

This is satisfied by the columns of any arithmetically Cohen generic subset of $\mathbb{N}\times\mathbb{N}$.\\

With these three observations, we prove:

\begin{prop}
For any reflexive binary relation $R$, there exists a subset $S$ of $\mathbb{R}$ such that $\langle\mathbb{N},R\rangle$ is isomorphic to $S$ paired with relative generic computation.

\end{prop}

\begin{proof}
By Observarion \ref{obs1}, it suffices to prove the proposition for recursive relations.

To prove this, we first define the asymmetric join of two reals in the following way. Fix any recursive partition of $\mathbb{N}$ into two parts, a density-0 part, $S$, and an infinite density-1 part, $L$. In particular, let $S$ be the powers of 2 and $L$ be the numbers that are not powers of 2. Then the asymmetric join of $A$ and $B$ is equal to $\mathcal{R}(A)$ on $L$, and has $B$ coded in to $S$ ($2^n$ is in the asymmetric join of $A$ and $B$ iff $n\in B$).

Note then that the asymmetric join of $A$ and $B$ has the same Turing degree as the usual join (it computes $B$, and it computes a density-1 subset of $\mathcal{R}(A)$). However, the ability to generically compute this asymmetric join is equivalent to the ability to compute of $A$. This is because a a density-1 subset of the bits of the generic join must include a density-1 subset of the bits of $\mathcal{R}(A)$, and $A$ can be computed from any density-1 subset of the bits of $\mathcal{R}(A)$. Conversely, if one could compute $A$, then one could compute $\mathcal{R}(A)$, and therefore one could compute the bits of the asymmetric join on all of $L$, yielding a generic computation of the asymmetric join of $A$ and $B$.\\


To conclude the proof, fix a recursive relation $R$. Choose $\langle X_i\rangle$ as in Observation $\ref{obs3}$, and for each $i$, define $Y_i$ to be the recursive join of the set of $X_j$ such that $R(i,j)$. (More precisely $Y_i=\lbrace\langle n,j\rangle|n\in X_j\wedge R(i,j)\rbrace$.) Let $Z_i$ be the asymmetric join of $X_i$ and $Y_i$.


If $R(i,j)$ then $Y_i\geq_TX_j$, so $Z_i\geq_TX_j$, so in particular we have that $Z_i$ generically computes $Z_j$. Likewise, if $\neg R(i,j)$ then $Z_i$ is computable from the join of the $X_k$ with $k\neq j$, so in particular, $Z_i\ngeq_TX_j$, so $Z_i$ does not generically compute $Z_j$.


\end{proof}

\section{No Minimal Pair}

Despite this fact, however, generic computation admits a worthwhile analysis in terms of what can be computed from a real. Certainly, if $X\geq_TY$, then the set of things that $X$ generically computes contains the set of things $Y$ generically computes, but the lack of transitivity does not allow one to draw the obvious conclusions one would like to draw. Indeed, it turns out that the reals have the recursion theoretically counterintuitive property that there are no minimal pairs for generic computation:

\begin{thm}\label{main}
For any nonrecursive reals $A$ and $B$ (and thus for any non-generically-computable reals $A$ and $B$), there exists a real $C$ such that $C$ is not generically computable, but such that $C$ is both generically $A$-computable and generically $B$-computable.

\end{thm}

The method for proving this statement will be to prove that for any nonrecursive reals $A$ and $B$, there are density-1 subsets of $\mathbb{N}$ recursive in $A$ and $B$ such that the union of the two subsets has no density-1 r.e. subset. Then $C$ will be the union of those two subsets. This will suffice by the following lemma:

\begin{lem}
Let $X$ be a density-1 subset of $\mathbb{N}$. Then for any real $A$, $A$ can generically compute $X$ if and only if $A$ can enumerate a density-1 subset of $X$.
\end{lem}

\begin{proof}
If $A$ can enumerate a density-1 subset $Y$ of $X$, then $A$ generically computes $X$ via the partial function $\varphi(n)=1$ iff $n\in Y$. The domain of this function is $Y$, which has density 1, and all the information it gives about $X$ is correct.

Conversely, if $A$ generically computes $X$, choose $\varphi$ such that $\varphi^A$ is a generic computation of $X$. Then let $Y=\lbrace n\ |\ \varphi^A(n)=1\rbrace$. $Y$ is clearly enumerable from $A$. Also, $Y$ is a subset of $X$ because a generic computation is not allowed to give incorrect answers. Finally $Y$ is density-1, because it is the intersection of two density-1 sets ($X$, and the domain of $\varphi^A$).
\end{proof}

The proof of Theorem \ref{main} will be comprised of three parts:

First, we prove Proposition \ref{1}, that if neither $A$ nor $B$ is $\Delta_2^0$, then $A$ and $B$ do not form a minimal pair for generic computability. Then, with the help of a technical lemma, we generalize this proof to prove Propositions \ref{2} and \ref{3}, the corresponding results for when one, or both of them are $\Delta_2^0$. Proposition \ref{3} is already proved in a paper of Downey, Jockusch, and Schupp, but the technical lemma that we use to prove Proposition \ref{2} proves Proposition \ref{3} as well.\\

We begin by introducing some terminology that will be used for the proofs. Let $P_i=\lbrace n\in\mathbb{N}\ \ |\ \ 2^i\leq n<2^{i+1}\rbrace$. Note then that $\mathbb{N}$ is the disjoint union of the $P_i$ together with $\lbrace0\rbrace$. For $X\subseteq\mathbb{N}$, we say that $X$ has a gap of size $2^{-e}$ at $P_i$ if the last $2^{i-e}$ many elements of $P_i$ are not elements of $X$. Note then the following lemma:

\begin{lem}\label{gaps}
If the only elements missing from $X$ are from gaps of the form just described, then $X$ is density-1 if and only if for every $e$, $X$ has only finitely many gaps of size $2^{-e}$
\end{lem}

\begin{proof}
If $X$ has a gap of size $2^{-e}$ at $P_i$ then $\frac{|X\upharpoonright2^i|}{2^i}\leq1-2^{-e-1}$, so in particular, if $X$ has infinitely many gaps of size $2^{-e}$, it does not have density 1.

Conversely, if there is some $i$ such that after $P_i$, all of the gaps in $X$ have size $\leq2^{-e}$, then for $n\geq2^{i+1}$, the density of $X\upharpoonright n$ will always be $\geq1-2^{-e+1}$. (Note that since the gaps appear at the ends of the $P_i$, the local minima of the density of $X$ always occur at the end of a $P_i$.) If for every $e$, there exists such an $i$, then the limiting density of $X$ will be 1.
\end{proof}

Before proving any of the propositions, we prove a result that is a direct corollary of the main theorem, and that also will not be used in the proof of the theorem. The proof is short though, and the methods generalize to prove Propositions \ref{1}, \ref{2}, and \ref{3}.

\begin{prop}\label{5}
For any nonrecursive real $A$, there is a density-1 set $\varphi^A$ that is recursive in $A$ such that $\varphi^A$ has no density-1 r.e. subset.
\end{prop}

First, a brief overview: we will define a total Turing functional $\varphi$ on $2^\omega$. For each $e$, there will be a strategy that diagonalizes against $W_e$ being a density-1 subset of $\varphi^X$ for any real $X$. In doing so, the $e$th strategy will cause at most one real $X_e$ to have the property that $\varphi^{X_e}$ is not density-1. $X_e$ will be the leftmost path through a recursive tree $T_e$, so repeating the same argument with rightmost paths gives a pair of functionals $\varphi$ and $\psi$ such that for any nonrecursive $X$ either $\varphi^X$ or $\psi^X$ is density-1.

\begin{proof}
Over the course of the construction, after stage $s$, $\varphi^X\upharpoonright2^s$ will be defined for every $X$, using at most the first $s$ bits of $X$. This will be useful in verifying that the strategies work as they are supposed to.
\\

The $e$th strategy acts as follows:

Define $T_e$ as the tree whose infinite paths consist of the reals $X$ such that $W_e\subseteq\varphi^X$. Note that $T_e$ is a recursive tree, since we can determine the $l$th level of $T_e$ in the following way. Run the enumeration of $W_e$ for the first $l$ steps. For every $n$ less than $2^l$, if $W_e$ has enumerated $n$, remove any $\sigma$ of length $l$ such that $\varphi^\sigma(n)=0$. As long as $l$ is less than the stage $s$ of the construction, this can be accomplished recursively. Let $T_{e,s}$ be the tree consisting of all extensions of the $(s-1)$th level of $T_e$.


At stage $s$, if $s<e$, do nothing. Also, if $T_{e,s}$ is finite, do nothing. Else, place a marker $p_s$ on the leftmost infinite path of $T_{e,s}$, at the shortest $\sigma$ on that path such that $\sigma$ has no marker. Then define $\varphi^X\upharpoonright P_s$ for all $X$ by placing a gap of size $2^{-e}$ into $\varphi^X$ at $P_s$ if $\sigma\prec X$, and by not placing such a gap if $\sigma\nprec X$.

The idea here is that the marker $p_s$ signifies the existence of a trap at $P_s$: $W_e$ must either avoid enumerating any of the elements of the gap at $P_s$, thereby creating another instance of its density dropping below $1-2^{e+1}$, or it must enumerate some of those elements, thereby removing all extensions of $\sigma$ from the tree $T_e$.

Note then that if $T_e$ is finite, then for every $X$, $\varphi^X$ has only finitely many gaps of size $2^{-e}$. Furthermore, in this case, the strategy has guaranteed that $W_e\nsubseteq\varphi^X$ for any $X$. (After the stage at which the tree is seen to be finite, the $e$th strategy stops acting. This stage is precisely the point at which, for every $X$, $W_e$ has enumerated something not in $\varphi^X$.)

On the other hand, if $T_e$ is infinite, then the leftmost path, $X_e$, of $T_e$ has infinitely many markers on it, so $\varphi^{X_e}$ has infinitely many gaps of size $2^{-e}$ so in particular, $W_e$ is not density-1. ($W_e\subset\varphi^X$ for every infinite path $X$ through $T_e$. This is because of how $T_e$ is defined.) Furthermore, if $X\neq X_e$ then $X$ has only finitely many markers on it, so $\varphi^X$ has only finitely many gaps of size $2^{-e}$ (If $X\neq X_e$ then there is some stage $s$ and some $\sigma\prec X$ such that after stage $s$, $\sigma$ never looks like it might be an initial segment of the leftmost path of $T_e$, so after that stage $s$, $X$ can only get at most $|\sigma|$ many markers placed on it)

Finally, note that the strategies have no need to interact: they ignore each other's markers and trees, and at state $s$, at most $s+1$ many strategies are eligible to act, and $\varphi^X$ gets defined on $P_s$ for every $X$ by just defining $\varphi(X)$ to be the intersection of the sets that each strategy wants $\varphi^X$ to be. By Lemma \ref{gaps}, a real $X$ will have the property that $\varphi^X$ is not density-1 if and only if some specific strategy causes $\varphi(X)$ to not be density-1. Thus, the only reals $X$ such that $\varphi^X$ is not density-1 are the $X_e$.

As mentioned in the overview, repeating the same construction again with rightmost paths will finish the proof, since if $X$ is the leftmost path of one recursive tree and the rightmost path of another, then $X$ is recursive. If this is not the case, then for one of the two constructions, the set computed from $X$ is density-1 and has no density-1 r.e. subset.

\end{proof}

Next, we proceed to prove Proposition \ref{1}. The proof is effectively the same as the proof of Proposition \ref{5}, with the only major modification being that we define two functionals simultaneously, and replace $T_e$ with a 4-ary branching tree whose paths correspond to pairs of reals $\langle X,Y\rangle$ such that $\varphi^X\cup\psi^Y\subset W_e$.

\begin{prop}\label{1}
If $A$ and $B$ form a minimal pair for generic computability, then either $A$ or $B$ is $\Delta_2^0$.
\end{prop}

\begin{proof}
Again, we describe how the $e$th strategy acts at stage $s$:

If $s<e$, do nothing. Otherwise define $T_{e,s}$ as the set of pairs $\langle\sigma,\tau\rangle$ such that $|\sigma|=|\tau|$ and such that $W_{e,s}\upharpoonright 2^{s-1}\subset\varphi_{s-1}^X\cup\psi_{s-1}^Y$ for some $X\succ\sigma$, and some $Y\succ\tau$. Note that this will be recursive, since we define $\varphi$ and $\psi$ applied to $P_s$ by the end of stage $s$.

Then, if $T_{e,s}$ is finite, do nothing. Otherwise put a marker $p_s$ on the leftmost infinite path of $T_{e,s}$, at the shortest pair $\langle\sigma,\tau\rangle$ on that path such that $\langle\sigma,\tau\rangle$ has no marker. Then define $\varphi_s^X\upharpoonright P_s$ and $\psi_s^Y\upharpoonright P_s$ for all $X$ and $Y$ by placing a gap of size $2^{-e}$ into $\varphi_s^X$ at $P_s$ if $\sigma\prec X$, and by not placing such a gap if $\sigma\nprec X$, and likewise placing a gap of size $2^{-e}$ into $\psi_s^Y$ at $P_s$ if and only if $\tau\prec Y$.

Note, as before, that only one path $T_e$ gets infinitely many markers on it, so in particular only finitely many markers are placed on nodes that are not on that path. That path corresponds to a pair of reals $\langle X_e,Y_e\rangle$, and if $X\neq X_e$ then $\varphi^X$ will have only finitely many gaps of size $2^{-e}$. Note also that the leftmost path of $T_e$ computes both $X_e$ and $Y_e$, so in particular, both are $\Delta^0_2$, but it is not true that either is necessarily the leftmost path of a recursive tree, so we are unable to use the previous trick to get them to be recursive.

Again, the strategies do not interfere with each other, and so if $A$ is different from all of the $X_e$ and $B$ is different from all of the $Y_e$ then $\varphi^A$ is density-1, $\psi^B$ is density-1, and $\varphi^A\cup\psi^B$ has no density-1 r.e. subset. By the comment after the statement of Theorem \ref{main}, this suffices.

\end{proof}

Now we prove our technical lemma, which states that the ``leftmost path'' in the above construction can be replaced by any uniformly chosen $\Delta_2^0$ infinite path through $T_e$.

\begin{lem}\label{6}
Let $\mathcal{F}$ be any function from reals to reals such that for a 4-ary branching tree $T$, if $T$ is infinite, then $\mathcal{F}(T)$ is an infinite path through $T$, and such that $\mathcal{F}(T)$ is uniformly recursive in $T^\prime$. Then for any reals $A$ and $B$, one of the following three things holds.

1: $A$ and $B$ do not form a minimal pair for generic computability.

2: There exists a recursive tree $T$ such that $\mathcal{F}(T)\geq_TA$.

3: There exists a recursive tree $T$ such that $\mathcal{F}(T)\geq_TB$.

\end{lem}

So, for example, letting $\mathcal{F}$ be the function corresponding to the construction in the proof of the low basis theorem, we could prove that if $A$ and $B$ form a minimal pair for generic computability, then either $A$ or $B$ must be low.


\begin{proof}
We modify the proof of Proposition \ref1 by, for each $e$, choosing a $\Delta_2^0$ index for $\mathcal{F}(T_e)$. This can be done uniformly by the fixed point theorem, since we can uniformly compute the trees $T_e$ from the graphs of $\varphi$ and $\psi$, and since $\mathcal{F}$ is assumed to be uniform. (By the fixed point theorem, we may assume we have a fixed index for the graphs of the Turing functionals $\varphi,\psi$ that we build.The graphs are recursive, as at stage $s$, $\varphi^X(n), \psi^Y(m)$ are defined for all $X$ and $Y$ and for all $n,m\leq2^s$.) 

Having chosen a $\Delta_2^0$ index for each $\mathcal{F}(T_e)$, at each stage, instead of placing a marker on the shortest unmarked node of the leftmost infinite path of $T_e$, the $e$th strategy places a marker on the shortest unmarked node of the current approximation to $\mathcal{F}(T_e)$. We then proceed to place the corresponding gaps in $\varphi$ and $\psi$ in the usual way. The $e$th strategy does nothing if $T_{e,s}$ is finite.

Then, as before, if $T_{e,s}$ is infinite, then the only path that gets infinitely many markers is $\mathcal{F}(T_e)$. This is because for any $n$, after the first $n$ bits of $\mathcal{F}(T_e)$ have stabilized to their final configuration, all future markers will be on extensions of $\mathcal{F}(T_e)\upharpoonright n$. In this case, $W_e$ is not density-1, so the $e$th strategy succeeds. If $T_{e,s}$ is finite, then for any $X$ and for any $Y$, $\varphi^X\cup\psi^Y\nsubseteq W_e$ as before.

Then, for each $e$, define $X_e$ and $Y_e$ as before, note that $\mathcal{F}(T_e)$ computes either of them, and note that $\varphi^X$ and $\psi^Y$ are both generic computations of the same non-generically computable real, as long as for every $e$, $X\neq X_e$ and $Y\neq Y_e$.

\end{proof}

We now can prove Proposition \ref{2} as a direct corollary of this lemma

\begin{prop}\label{2}
If $A$ is $\Delta_2^0$ and $B$ is not $\Delta_2^0$ then $A$ and $B$ do not form a minimal pair for generic computability.
\end{prop}

\begin{proof}
For any infinite recursive tree $T$, and any nonrecursive $\Delta_2^0$ set $A$, $0^\prime$ can uniformly (in indices for $T$ and $A$) compute an infinite path $Z\in[T]$ such that $Z\ngeq_T A$.

Apply Lemma \ref6 using the $\mathcal{F}$ representing this computation, and note then that for any $T$, $\mathcal{F}(T)\ngeq A$ by construction. Also, $\mathcal{F}(T)\ngeq B$ because $B$ is not $\Delta_2^0$, and $\mathcal{F}(T)$ is. Therefore, $A$ and $B$ do not form a minimal pair for generic computability.




\end{proof}

Note that this same proof can be modified to prove Proposition $\ref3$, by simultaneously avoiding the cones above both $A$ and $B$.

\begin{prop}\label{3}
\rm (Downey, Jockusch, and Schupp)
\it
If $A$ and $B$ are both $\Delta_2^0$ then $A$ and $B$ do not form a minimal pair for generic computability.
\end{prop}

\section{Generic Reducibility}

Given the wildly nontransitive nature of generic computation, it seems natural to attempt to generalize it to some transitive notion of reducibility with the same basic properties. The most important properties that one might expect of it would be to preserve the definition of the generically computable sets, and in particular a generic reduction using a generically computable oracle should be the same as a generic computation. In this section, we introduce four notions of generic reducibility, prove that two of them are equivalent, and prove that at least two of them are distinct. We then analyze the implications of Theorem \ref{main} towards these notions of generic reducibility.

\begin{defn}
\rm
A \it generic description \rm of a real $A$ is a set $S$ of ordered pairs $\langle n,x\rangle$, with $n\in\N, x\in\lbrace 0,1\rbrace$, such that:

if $\langle n,0\rangle\in S$ then $n\notin A$,

if $\langle n,1\rangle\in S$ then $n\in A$,

and $\lbrace n\ |\ \exists x\langle n,x\rangle\in S\rbrace$ is density-1.

\end{defn}

It should be mentioned that this notation conflicts slightly with the notation of Jockusch and Schupp, in that they define a generic description as a partial function, and this would be the graph of a generic description, by their definition. For the purposes of generic reduction though, it is more useful to have ``generic description'' be defined as a set, since the output of a generic computation is a generic description, and so the input of a generic reduction should be a generic description. It is interesting to notice, though, that the output of a generic computation of $A$ is more than just a generic description of $A$, in that it is an \it enumeration \rm of a generic description of $A$. For this reason, we define:

\begin{defn}
\rm
A \it time-dependent generic description \rm of a real $A$ is a set $S$ of ordered triples $\langle n,x,l\rangle$, with $n,l\in\N, x\in\lbrace 0,1\rbrace$, such that $\lbrace\langle n,x\rangle\ |\ \exists l\langle n,x,l\rangle\in S\rbrace$ is a generic description of $A$.

\end{defn}

Then, $B$ generically computes $A$ if and only if $B$ can enumerate a generic description of $A$, which is true if and only if $B$ can compute a time-dependent generic description of $A$.

Jockusch and Schupp define generic reducibility via enumeration operators. The relation is basically the one that one might intuitively construct, using only densely much information from the oracle, $A$, to deduce a generic computation of $B$. For those familiar with the notation, the definition is as follows:

\begin{defn}
\rm
A \it generic reduction \rm of $B$ from $A$ is an enumeration operator which, given any generic description of $A$ as input, outputs a generic description of $B$. $B$ is \it generically reducible to $A$ \rm if there exists a generic reduction of $B$ from $A$. In this case, we write $A\geq_GB$.
\end{defn}

For those unfamiliar with the notation,

\begin{defn}
\rm
An \it enumeration operator \rm is an r.e. set $W$ of codes for pairs $\langle n,D\rangle$ where $n\in\N$ and $D$ is a code for a finite subset of $\N$. It is thought of as a function, sending a real $A$ to the set $\lbrace n\ |\ \exists D[D\subseteq A\wedge\langle n,D\rangle\in W]\rbrace$. 
\end{defn}

It should be noted that generic reduction has the following features: The computation of $B$ from $A$ must be uniform in the generic description of $A$, and the computation is only allowed to reference which sets are contained in the graph of the input set when computing a generic description for the output set (so in particular, giving less information about $A$ never results in more information about $B$, and information is not allowed to be deduced from the rate/order of enumeration of the graph of  $A$).

With these comments in mind, we make the following definitions:

\begin{defn}
\rm
A \it time-dependent generic reduction \rm of $B$ from $A$ is a Turing functional which, given any time-dependent generic description of $A$ as input, outputs a time-dependent generic description of $B$. $B$ is \it time-dependently generically reducible to $A$ \rm if there exists a time-dependent generic reduction of $B$ from $A$.
\end{defn}

\begin{defn}
\rm
$B$ is \it non-uniformly generically reducible to $A$ \rm if for every generic description of $A$, there is an enumeration operator which outputs a generic description of $B$ using the given generic description of $A$ as input.
\end{defn}

\begin{defn}
\rm
$B$ is \it non-uniformly time-dependently generically reducible to $A$ \rm if every time-dependent generic description of $A$, can compute a time-dependent generic description of $B$.
\end{defn}

Again, we mention that the ability to compute a time-dependent generic description of a set is equivalent to the ability to enumerate a generic description of the set, and so the difference between the time-dependent and non-time-dependent reductions is entirely a difference in terms of what the input of the reduction is. The outputs are phrased differently just to make transitivity obvious, and also to make it easier to work with any given form of reduction on its own. Also for the remainder of this paper, whenever a non-time-dependent generic description is used as an oracle, only positive information about elements in the description will be used for the computation.

From the definitions, we may immediately conclude the following implications.

\begin{obs}\label{uniform}
The existence of a uniform reduction of either type implies the existence of a non-uniform reduction of the corresponding type.
\end{obs}

\begin{obs}\label{timedep}
The existence of a non-time-dependent reduction of either type implies the existence of a time-dependent reduction of either type.
\end{obs}

Observation \ref{uniform} is trivially true. Observation \ref{timedep} is true since from any time-dependent generic description, one can enumerate a generic description, and then simply ignore the order in which the elements were enumerated. The rules of enumeration operators do not allow for an obvious converse to this, although the first proposition that we prove is that the converse of does hold for the uniform generic reductions, and so in particular, the two uniform reductions are equivalent.

\begin{prop}\label{time}
$A\geq_GB$ if and only if the following holds:

There is a Turing functional $\varphi$ such that for any time-dependent generic description $X$ of $A$, $\varphi^X$ is a generic computation of $B$.

\end{prop}

By Observation \ref{timedep}, we need only prove that the second implies the first.

\begin{proof}
Assume that from every time-dependent generic description $X$ of $A$,  $\varphi^X$ is a generic computation of $B$. Then in particular, there are no time-dependent generic descriptions $X$ such that $\varphi^X$ gives false information about $B$. So to generically reduce $B$ to $A$, one needs only to enumerate everything that $\varphi$ would enumerate, given any ordering of the generic description.

In other words, let $W$ be the set of all $\langle a,D\rangle$ such that $a$ codes an ordered pair $\langle x,y\rangle$, $D$ is a finite set of ordered pairs $\langle n_i,m_i\rangle, i\leq c$, and such that there exists a sequence $\langle l_i\ |\ i\leq c\rangle$ where $\varphi^X$ gives output $y$ on input $x$ for some $X$ extending $\lbrace\langle\langle n_i,m_i\rangle,l_i\rangle\ |\ i\leq c\rbrace$.

Then, for any generic description of $A$, the output of $W$ will be the union of all outputs of $\varphi^X$ for any time-dependent versions of that generic description, and in particular, will be a generic description of $B$.
\end{proof}

Next, we show that neither of the nonuniform reductions is equivalent to the uniform reduction. To prove this, we introduce some notation. Recall from Section 1, the definition of $\mathcal{R}(X)$:

\begin{defn}
\rm
For any real $X$, $n\in \mathcal{R}(X)\leftrightarrow m\in X$, where $2^m$ is the largest power of $2$ dividing $n$.
\end{defn}

Note now the following strengthening of Observation \ref{R}, proved by Jockusch and Schupp \cite{stuff}, but with a proof included for completeness:

\begin{obs}\label{R2}
For any real $X$, $X$ computes $\mathcal{R}(X)$ uniformly and $X$ can be computed uniformly from any generic description of $\mathcal{R}(X)$. Therefore, the map sending $X\mapsto\mathcal{R}(X)$ induces an embedding from the Turing degrees to the generic degrees.
\end{obs}

\begin{proof}
$X$ computes $\mathcal{R}(X)$ uniformly, and so generically computes $\mathcal{R}(X)$ uniformly.

Conversely, to compute the $m$th bit of $X$ from a generic description of $\mathcal{R}(X)$, search for any $n$ such that $2^m$ is the largest power of 2 dividing $n$ and such that the generic description of $\mathcal{R}(X)$ has a value for the $n$th bit of $\mathcal{R}(X)$. Use that as the value for the $m$th bit of $X$. There must be such an $n$ because the set of numbers divisible by $2^m$ and not by $2^{m+1}$ has positive density (in fact, has density $\frac1{2^{m+1}}$), and the generic description has values for density-1 many bits of $\mathcal{R}(X)$.

The proof of the embedding follows directly: If $X$ computes $Y$ then any generic description of $\mathcal{R}(X)$ can be used uniformly to compute $X$ and therefore $Y$ and therefore $\mathcal{R}(Y)$. Likewise, if $\mathcal{R}(X)\geq_G\mathcal{R}(Y)$, then $X$ can compute $\mathcal{R}(X)$, which can generically compute $\mathcal{R}(Y)$. $Y$ can then be recovered from this generic description.
\end{proof}

Note that this observation and proof hold for any of the forms of generic reduction.

We now introduce an alternate definition, which would have sufficed for the purposes of Observation \ref{R}, but with the property that one cannot uniformly recover $X$ from a generic description of $\widetilde{\mathcal{R}}(X)$.

\begin{defn}
\rm
For any real $X$, $n\in \widetilde{\mathcal{R}}(X)\leftrightarrow m\in X$, where $2^m$ is the largest power of $2$ less than $n$.
\end{defn}

The key distinction here is that $\mathcal{R}(X)$ codes each of the entries of $X$ into a positive density set, and so any generic description of $\mathcal{R}(X)$ must be able to recover all the entries of $X$, while $n\in \widetilde{\mathcal{R}}(X)\leftrightarrow m\in X$ codes the entries of $X$ into progressively larger sets, each finite, and so any generic description of $n\in \widetilde{\mathcal{R}}(X)\leftrightarrow m\in X$ must be able to recover all but finitely many of the entries of $X$. 

\begin{prop}\label{nonequivalence}

Let $A$ be a real such that one cannot uniformly compute $A$ from an arbitrary cofinite subset of the entries of $A$. Then $\mathcal{R}(A)$ is non-uniformly generically equivalent to $\widetilde{\mathcal{R}}(A)$ (and therefore non-uniformly time-dependently generically equivalent), but $\widetilde{\mathcal{R}}(A)\ngeq_G\mathcal{R}(A)$. Therefore, neither of the non-uniform notions of generic reducibility is equivalent to the uniform notion.

\end{prop}

Note that there exists reals satisfying the hypothesis of the proposition, for example any 1-random real, or any arithmetically Cohen generic real.

\begin{proof}

First of all, any generic description of $\mathcal{R}(A)$ can be used uniformly to recover $A$ by Observation \ref{R2}. It can therefore compute $\widetilde{\mathcal{R}}(A)$. Likewise, from any generic description of $\widetilde{\mathcal{R}}(A)$, one can uniformly compute all but finitely many bits of $A$. From this, one can non-uniformly compute $A$, and therefore compute $\mathcal{R}(A)$. However, this cannot be done uniformly, as follows:

Any cofinite subset of the entries of $A$ can be used to uniformly compute a generic description of $\widetilde{\mathcal{R}}(A)$. Any generic description of $\mathcal{R}(A)$ can be used to uniformly compute $A$. Thus, since there is no uniform way to compute $A$ from a cofinite subset of its entries, there is no uniform way to go from a generic description of $\widetilde{\mathcal{R}}(A)$ to a generic description of $\mathcal{R}(A)$.
\end{proof}

As yet, however, there do not appear to be any theorems that have been proved for one form of generic reduction that have not also been proved for every other form. Indeed, for the remainder of this paper generic reduction will refer to any of the definitions.

Theorem \ref{main} sheds very little light on the question of the existence of a minimal pair in the generic degrees. It does, however, strengthen a previous result of Jockusch and Schupp. Borrowing a term from the study of enumeration reducibility, we make the following definition:

\begin{defn}
\rm
A generic degree \bf a \rm is \it quasi-minimal \rm if \bf a \rm is nonzero, and if for every nonrecursive real $X$, the generic degree of $\mathcal{R}(X)$ is not below \bf a\rm.
\end{defn}

Then, while proving that the embedding from the Turing degrees to the generic degrees is not surjective, Jockusch and Schupp \cite{stuff} actually prove the stronger result that there exists a quasi-minimal generic degree. Theorem \ref{main} allows us to strengthen this result: 

\begin{prop}\label{qm}
For every nonzero generic degree \bf a\rm, there is a quasi-minimal generic degree \bf b \rm such that $\bf a\it\geq_G\bf b\rm$ .
\end{prop}

\begin{proof}
Let \bf a \rm be a nonzero generic degree. If \bf a \rm is quasi-minimal, then we are done. Else, choose $A$ nonrecursive, so that the generic degree of $\mathcal{R}(A)$ is below \bf a\rm. In the Turing degrees, every nonzero degree is half of a minimal pair, so choose $B$ such that $A$ and $B$ form a minimal pair for Turing reducibility. By Theorem \ref{main}, choose some $C$, not generically computable, such that $A$ and $B$ can both generically compute $C$.

Then $C$ generically reduces to both $\mathcal{R}(A)$, and $\mathcal{R}(B)$, and the generic degree of $C$ must be quasi-minimal, since if there were some $D$ such that $C\geq_G\mathcal{R}(D)$, then we would have $\mathcal{R}(A)\geq_G\mathcal{R}(D)$ and also $\mathcal{R}(B)\geq_G\mathcal{R}(D)$, and therefore $A\geq_TD$ and also $B\geq_TD$, and so $D$ would have to be recursive, since $A$ and $B$ form a minimal pair. Let  \bf b \rm be the generic degree of $C$.
\end{proof}

Unfortunately, it seems very difficult to generalize the methods of this paper to prove the lack of a minimal pair for generic reducibility, since if there is such a pair, then at least one of the two degrees in the pair would have to be quasi-minimal, and the construction in the proof of Theorem \ref{main} involves each of $A$ and $B$ actually computing something which is a generic description of $C$, but the quasi-minimal generic degrees are unable to use a generic reduction to completely compute any nonrecursive real. However, Proposition \ref{qm} gives some hope of a different construction, since it ensures that if there exists a minimal pair for generic reducibility, then there exists one in which both halves of the minimal pair are quasi-minimal.

Working in the other direction, it would also be interesting to know whether two quasi-minimal generic degrees can join to a generic degree that is not quasi-minimal.


\begin{thebibliography}{aa}

\bibitem[1]{stuff}
Carl Jockusch and Paul Schupp, \emph{Generic computability, {T}uring degrees,
  and asymptotic density}, to appear in \emph{Journal of the London Mathematical Society}.



\end{thebibliography}
\end{document}